\documentclass{article}
\usepackage{arxiv}
\usepackage{algorithm}
\usepackage{algpseudocode}
\usepackage[utf8]{inputenc} % allow utf-8 input
\usepackage[T1]{fontenc}    % use 8-bit T1 fonts
\usepackage{hyperref}       % hyperlinks
\usepackage{url}            % simple URL typesetting
\usepackage{booktabs}       % professional-quality tables
\usepackage{amsfonts}       % blackboard math symbols
\usepackage{nicefrac}       % compact symbols for 1/2, etc.
\usepackage{microtype}      % microtypography
\usepackage{lipsum}
\usepackage{graphicx}
\usepackage{amsmath}
\usepackage{amssymb}
\usepackage{amsthm}
\graphicspath{ {./images/} }
\newtheorem{theorem}{Theorem}

\newtheorem{corollary}{Corollary}

\newtheorem{definition}{Definition}
\newtheorem{example}{Example}

\newtheorem{lemma}{Lemma}

\newtheorem{proposition}{Proposition}
\newtheorem{remark}{Remark}

\usepackage{verbatim}

\usepackage{amssymb,tikz}
\usetikzlibrary{positioning}
\tikzset{sgplattice/.style={inner sep=1pt,norm/.style={red!50!blue},char/.style={blue!50!black},
  lin/.style={black!50}},cnj/.style={black!50,yshift=-2.5pt,left=-1pt of #1,scale=0.5,fill=white}}

\title{A survey of the number of supersolvable subgroups of finite groups}

\author{
 Primitivo B. Acosta-Hum\'anez \\
 Instituto de Matem\'atica \& Escuela de Matem\'atica\\
Universidad Aut\'onoma de Santo Domingo \\
  Dominican Republic \\
  \texttt{pacosta-humanez@uasd.edu.do} \\
   \AND
  Orieta Liriano \\
 Instituto de Matem\'atica \& Escuela de Matem\'atica\\
Universidad Aut\'onoma de Santo Domingo \\
Ciclo de Estudios Generales\\
Universidad Iberoamericana\\
Dominican Republic \\
  \texttt{oliriano25@uasd.edu.do} \\
  \And
Francis Mora-Ferreras \\
Instituto de Matem\'atica\\
Universidad Aut\'onoma de Santo Domingo \\
Dominican Republic \\
Instituto de Matem\'aticas\\
Universidad de Talca, Chile\\
\texttt{fmora@inst-mat.utalca.cl} \\
  }

\begin{document}
\maketitle
\begin{abstract}
  {In this paper we survey a new criteria for solvability of finite groups in terms of number of supersolvable (also known as polycyclic) and non-supersolvable subgroups. In particular, we present original examples of supersolvable groups such as non-abelian groups with two cyclic generators. Additionally, we present some examples and applications in GAP system providing some description about recent criteria for solvability and properties of finite groups based on the number of supersolvable and non-supersolvable subgroups of a finite group $G$. }
\end{abstract}
\keywords{Finite groups\and  Polycyclic groups \and Solvable groups \and Supersolvable groups }

\tableofcontents

\section{Introduction}
Since the creation of the theory of groups, it is known as one of the most abstract theories of mathematics. However, owing to the development of computational applications to deal with problems associated with group theory, it has had a great boom in recent decades. Some researchers called this combination of group theory and computational applications as \emph{Computational Group Theory}. Combinatorial aspects of Group Theory, such as the computation of the number of subgroups of certain type,  has been developed by some researchers in recent decades and they called it \emph{Combinatorial Group Theory}. In this sense, this paper falls in the topics Combinatorial Group Theory and Computational Group Theory.\\

Using the free access tool \href{www.zbmath.org}{\emph{zentralblatt}} we found 585 documents on Combinatorial Group Theory, among which there are 272 papers, 226 chapter of books and 84 books. Likewise, a search was made for the theory of Computational Groups and 320 documents were found, among which are 139 papers, 150 chapter of books and 31 books. The intersection of the two sets is 14 documents, 1 book, and 14 papers in special volumes. Thus, the theory of finite groups is an important topic in mathematics because from the theoretical point of view it lead to new developments in other in mathematics and also it has been applied successfully in other areas such as physics, chemistry among others.\\

On the other hand, since the seminal work of E. Galois, in where the concept of \emph{Solvable Group}, which involves abelian groups, was introduced, see \cite{galois1830,galois1830note}, the solvability of groups has been deeply studied for a plenty of researchers in algebra, see \cite{feit1963solvability}. In the Century XX was introduced an intermediate category of groups between abelian and solvable groups, the so-called \emph{supersolvable} (see \cite{weinstein1982}), \emph{supersoluble} (\cite{baer1955}) or \emph{polycyclic} groups (see \cite{remeslennikov1969}).\\

In 1924 Schmidt \cite{schmidt1924} studied non nilpotent groups whose proper subgroups are nilpotent and proved that these groups are solvable. Such groups are now known as Schmidt groups. Rédei in \cite{redei1956endlichen} and Ballester-Bolinches, Esteban-Romero and Robinson in \cite{ballester2005finite} classify Schmidt groups. In \cite{ballester2005finite} they are called minimal non-nilpotent groups.\\

Then, later in 2012, Zarrin \cite{zarrin2012} extends Schmidt's result and characterizes the semi-simple groups $G$ with at most 65 non-nilpotent subgroups.  For them Zarrin introduces the class $\mathcal{S}^n$; which is the class of groups with exactly $n$ non-nilpotent subgroups, i.e. $G$ is an $\mathcal{S}^n$-group if it has n non-nilpotent subgroups or $G\in \mathcal{S}^n$. Zarrin proves the following results in \cite{zarrin2012}. Given $G$ a finite group that is a $\mathcal{S}^{n}$-group with $n\leq22$. Then $G$ is solvable and $G$ is isomorphic to alternating group of degree 5. Also if $G$ is a alternating group of degree 5, symmetric group of degree 5, or or the special linear group of degree 2 over a field $\mathbb{F}_7$ is the set of $2\times2$ matrices with determinant 1.\\

Based on Zarrin's results on \cite{zarrin2012}, he obtains that the only simple groups with exactly 22 non-nilpotent subgroups is alternating group of degree 5 and that the only simple group with exactly 65 non-nilpotent subgroups is the special linear group of degree 2 over a field $\mathbb{F}_7$. A corollary to this is that if a non-solvable group $G$ is an $S^{n}$-group with $n\leq 22$, then $G/\Phi(G)$ is isomorphic to alternating group of degree 5.\\

In this same paper Zarrin \cite{zarrin2012} conjectured that a finite non-solvable group that is a $S^{22}$-group is isomorphic to alternating group of degree 5 or the special linear group of degree 2 over a field $\mathbb{F}_7$.\\

This conjecture was proven by Ballester-Bolinches, Esteban-Romero and Lu in 2017 in \cite[Theorem A]{BRL2017}. That is, they proved the following result :
 given $G$ be a solvable group. Then $G$ has exactly 22 non-nilpotent subgroups if and only if it is isomorphic to alternating group of degree 5 or the special linear group of degree 2 over a field $\mathbb{F}_5$.\\

This papers is structured as follows. Section 2 contains some theoretical aspects concerning basic group theory, necessaries to understand the rest of the paper. Section 3 contains concepts supersolvable groups and some remarkable results which are well known in this theory. Section 4 contains some relevant aspects of supersolvable groups and Zarrin criteria to obtain solvable groups. Finally, in Section 5 we introduce new results concerning the supersolvability of non-abelian groups generated by two cyclic generators and we included some programs developed in GAP system.

\section{Preliminaries about classical Group Theory}

This section contains the basic theoretical background that every undergraduate student of algebra knows. We start introducing concepts such as groups, subgroups, normal subgroups, homomorphisms, solvable groups. We follows some classical books such as Our start point is the well known book written by Doerk K. and  Hawkes T., see \cite[pag. 1]{doerk2011}.\\

In this way, a group is a non-empty set $ G $ equipped with a binary operation which is associative $$g(hw)=(gh)w \quad \quad \text{for all} \ \ g,h,w \in G$$ (which we will usually denote multiplicative by juxtaposition) with the property that $G$ has an identity element $1 = 1_{G}$ that satisfies
$$
1 g = g 1 = g \quad \text {for all} \ \ g \in G
$$
and for each $ g $ of $ G $ has right and left inverse $ g ^ {- 1} $ satisfying
$$
g ^ {- 1} g = g g ^ {- 1} = 1
$$
If $ h g = g h $ for all $ g, h \in G $, we say that $ G $ is an {\bf  abelian group}. The order of $G$ is the cardinality of the set and is written $ | G | $. If $ | G | $ is finite, we say that $ G $ is finite.\medskip

Throughout this paper, $G$ will always denote {a finite group.}\medskip

	A subset $ U $ of $ G $ is called { \bf subgroup} if it is a group with respect to the binary operation defined on $ G $: for this we write $ U \leq G $ and $ U < G $ when $ U \neq G $. If $U < G$, we call $U$ a proper subgroup of $G$. We will use the symbol $1$ to denote the identity subgroup $ \{1 \} $ of a group. If $ U $ is a proper subgroup with the property that $ U = V $ whenever $ U \leq V < G $, we call $ U $ a {\bf maximal subgroup} of $ G $; therefore, the maximal subgroups are precisely the maximal elements of the set of partially ordered proper subgroups by inclusion of $G$. \medskip

If $ U \leq G $ and $ g \in G $, we write $ U g $ instead of $ U \{g \} $ and call $ U g $ a right {coset} of $ U $ in $ G $, the left cosets are defined analogously. Calling two elements $ g $ and $ h $ of $ G $ equivalent if and only if $ hg ^ {- 1} \in U $, we obtain an equivalence relation on $ G $ whose equivalence classes are exactly the right cosets of $ U $ in $ G $ (The set of right cosets of $ U $ in $ G $ will be denoted by $ G / U $). It follows that these right cosets form a partition of $G$ and, in particular, that there exists a subset $T$ of $G$ such that
$$
G=\bigcup_{t \in T} Ut, \quad \text { y } \quad Us \cap Ut=\varnothing
$$
provided that $ s, t \in T $ and $ s \neq t. $ This partition is called the right lateral decomposition of $ G $ by $ U. $ A set $ T $ of the type just described is called a right transversal of $ U $ into $ G $; is a set that contains exactly one element from each right coset, so the number of right cosets is $ \prod _ {t \in T} \left | U{t} \right | $
\begin{example}
Consider $$ G = D_ {8} :=
\left\langle s, r \mid r^{4}=s^{2}=1, s r s^{-1}=r^{-1}\right\rangle=\{1,r,r^2,r^3,s,sr,sr^2,sr^3\}
$$ the diedral group of order $8$ and consider $\langle r^2\rangle=\{1,r^2 \}$. \medskip
% y sea $ A = \left \{1, r, r ^ {2}, r ^ {3 } \right \} $  el subgrupo de rotaciones en $ D_ {8} $. 

Therefore, we observe $$
D_{8} /\left\langle r^{2}\right\rangle=\left\{g\left\{1, r^{2}\right\}: g \in D_{8}\right\}=\left\{\left\{1, r^{2}\right\},\left\{r, r^{3}\right\},\left\{s, s r^{2}\right\},\left\{s r, s r^{3}\right\}\right\}
$$
That is, there are $4$ left cosets. Moreover, such left cosets are explictly given by $$D_8=\left\{1, r^{2}\right\}\cup \left\{r, r^{3}\right\}\cup\left\{s, s r^{2}\right\}\cup\left\{s r, s r^{3}\right\}$$
\end{example}

There is a corresponding partition of $G$ into the left cosets of $U$, and a complete set of representatives of the left cosets is called a left traverse of $U$ into $G$.\medskip

The mapping $U g \rightarrow g ^{- 1} U $ is a bijection from the set $G/U$ of right cosets to the set of left cosets of $U$ in $G; $ the common cardinality of these two sets is called the index of $ U $ in $ G $ and is written $ | G: U | $.\medskip

If $ g \in G $, the mapping $ u \rightarrow ug $ is a mapping from $ U $ to the coset $ Ug$. Thus, all right cosets have the cardinality $ | U | $, and we get the following famous Lagrange theorem, see \cite[Theorem 1.4, pag. 3]{doerk2011}.

\begin{theorem}[Lagrange's Theorem] If $ U $ is a subgroup of a group $ G $, then $ | G | = | G: U || U | $. In particular, if $G$ is finite, $|U|$ is a divisor of $ | G | $.
\end{theorem}

From the previous example we have $|D_8|=|D_8:\langle r^2\rangle||\langle r^2\rangle|=|D_8/\langle r^2\rangle||\langle r^2\rangle|=4\cdot2$.\medskip

\begin{comment}
Si $ g, h \in G $, establecemos $ h ^ {g} = g ^ {-1} hg $, y si $ X $ es un subconjunto no vac \'io de $ G $, definimos $ X ^ {g } $ como el conjunto $ \left \{x ^ {g}: x \in X \right \}. $ Adem\'as usamos $ X ^ {G} $ para denotar el conjunto
$$
X^{G}:=\left\{X^{g}: g \in G\right\}
$$

de { conjugados} de $ X $ en $ G $. El mapeo $ \rho_ {g} : G \rightarrow G $ definido por

$$
\rho_{g}(h)=h^{g}
$$
para todo $ h \in G $ es visto como un automorfismo de $ G $; se llama { automorfismo interno} inducido por $ g. $ Un automorfismo $ \alpha $ de $ G $ se llama { $\operatorname{inner}$} si $ \alpha = \rho_ {g} $ para algunos $ g \in G $, y el conjunto de todos los automorfismos internos se denotan con $ \operatorname {Inn} (G)$. Evidentemente, { un grupo es abeliano} si y solo si $ \operatorname {Inn} (G) = 1. $ Pues si $G$ es abeliano y $g,h\in G$ se tiene que $$\rho_g(h)=h^g=g^{-1}hg=g^{-1}gh=h.$$

Un subgrupo $ N $ de $ G $ que es invariante bajo todos los automorfismos internos (para los cuales, por lo tanto, $ N ^ {g} = N $ para todos los $ g \in G $) se denomina { subgrupo normal} de $ G. $ If $ N $ es un subgrupo normal de $ G $, lo denotamos simb\'olicamente por $ N \unlhd G ($ y por $ N \triangleleft G $ cuando $ N \neq G $). 

Claramente, 1 y $ G $ son siempre subgrupos normales de $ G $, y un grupo $ G \neq 1 $ sin otros subgrupos normales se llama { simple}. Por tanto, $ G $ es simple si y solo si tiene precisamente dos subgrupos normales.
\end{comment}

If $ g $ is an element of a finite group $ G $, there exists a smallest natural number $n$ such that $g ^{n} = 1$; this is called the order of $g$ and is written $o(g)$. It is easy to see that $ o (g) = | \langle g \rangle | $, the order of the {cyclic subgroup} generated by $ g $, and therefore $ o (g) $ divides $| G | .$ The least common multiple of the integers $ \{o (g): g \in G \} $ is called the exponent of $ G $ and is written $ \operatorname {Exp} (G ) $.

\subsection{About $p$-groups and $p$-subgroups}

If $p$ is a prime number, we say that $G$ is a {\bf $p$-group} if $G$ is finite and $|G|=p^n$ for some $n\in \mathbb{N }.$
If we have a finite group $G$ and a prime number $p$, a {\bf $p$-subgroup} of $G$ is a subgroup $H$ of $G$ whose order is a power of $p \ \ (|H|=p^m \ \ \text{for some } \ m\in\mathbb{N})$. If $ | G | = p ^ {\alpha} m $ where $ p $ does not divide $ m $, then a subgroup $H$ of order $ p ^ {\alpha} $ is called a $ p$-Sylow subgroup of $ G$.\medskip

The set of Sylow $p$-subgroups is denoted by $\operatorname{Syl}_p(G)$. An almost immediate result is that if $H$ is a subgroup of a finite group $G$. Then $H\in\operatorname{Syl}_p(G)$ if and only if $|H|$ is a power of $p$ and also $|G:H|$ is not divisible by $p$ . \medskip

%Si $ \pi $ es un conjunto de n\'umeros primos y si cada divisor primo de $ o (g) $ pertenece a $ \pi $, llamamos a $ g $ un $ \pi$-elemento. Si cada elemento de un grupo $ G $ es un $ \pi$-elemento, llamamos a $ G $ un $ \pi $-grupo. El conjunto de primos distintos que dividen $ | G | $ se denota por $ \sigma (G) $, y del teorema de Sylow se deduce f\'acilmente que $ G $ es un $ \pi$-grupo  si y solo si $ \sigma (G ) \subseteq \pi. $ El conjunto complementario de primos, $ \mathbb {P}$ de $ \pi $, se indicar\'a con $ \pi ^ {\prime} $. Si $ \pi = \{p \} $, un singleton, hablamos de $ p$-elementos y $ p $-grupos y escribimos $ p ^ {\prime} $  en lugar de $ \{p \} ^ {\prime} $.

\begin{theorem}[Sylow. Existence \cite{navarro2017curso}]
    Let $G$ be a finite group. Then $\operatorname{Syl}_p(G)$ is not an empty set.
\end{theorem}
A corollary that follows from this theorem and is attributed to Cauchy is that if $G$ is a finite group and $p$ is a prime that divides $|G|$, then there exists $x\in G$ of order $p $. Also a result that is not very difficult to prove is that for $p$ prime and $p^a\mid |G|$, then there exists a subgroup $H$ of $G$ of order $p^a$, since if $P\in \operatorname{Syl}_p(G)$ then $p^a\mid |P|.$\medskip

\begin{corollary}[Sylow. Dominance, \cite{navarro2017curso}]
    Let $G$ be a finite group and let $p$ be a prime. If $H$ is a $p$-subgroup of $G$ it is contained in some $p$-Sylow of $G$.
\end{corollary}

The following corollary follows from the theorem that states that $G$ is a finite group and $p$ is a prime. If $H$ is a $p$-subgroup of $G$ and $P\in \operatorname{Syl}_p(G)$, then there exists $g \in G$ such that $H$ is a subset $P^g$.\medskip

\begin{corollary}[Sylow. Conjugacy \cite{navarro2017curso}]
  Consider $P,Q \in \operatorname{Syl}_p(G)$. Then there exists $g \in G$ such that $P = Q^g$
\end{corollary}

\subsection{Homomorphisms}
Let $G$ and $H$ be groups. A mapping $ \alpha: G \rightarrow H $ is called a homomorphism if
$$
\alpha(x y) = \alpha(x) \alpha(y)
$$
for all $ x, y \in G $.\medskip

We would not be observing any hard and fast rules about writing maps left or right, but will simply choose whichever side seems most appropriate in a given context. As usual, an injective homomorphism (respectively \emph{surjective, bijective}) is called a \emph{monomorphism} (respectively \emph{epimorphism, isomorphism}). Occasionally, the notation $ \alpha: G \mapsto H $ will mean a monomorphism and $ \alpha: G \rightarrow H $ an epimorphism. If there is an \emph{isomorphism} $ \alpha: G \rightarrow H $, we say that $ G $ is \emph{isomorphic} with $ H $ (or that $ G $ and $ H $ have the same type of isomorphism) and we write $ G \cong H $. \medskip

A homomorphism $ \alpha $ of $ G $ is likewise called an \emph{endomorphism} and when $ \alpha $ is bijective, an \emph{automorphism} of $ G. $ The set of all automorphisms of $ G $ forms a group under the binary operation composition of mappings; this is called the \emph{automorphism} group of $G$, is denoted by $\operatorname{Aut}(G)$, and is obviously a subgroup of $\operatorname{Sym}(G)$.\medskip

If $ g, h \in G $, we set $ h ^ {g} = g ^{-1} hg $, and if $ X $ is a non-empty subset of $ G $, we define $ X ^{g} $ to be the set $ \left \{x^{g}: x \in X \right \}. $ In addition we use $ X^{G} $ to denote the set
$$
X^{G}=\left\{X^{g}: g \in G\right\}
$$

from {conjugates} from $X$ into $G$. The mapping $ \rho_ {g} : G \rightarrow G $ defined by
$$
\rho_{g}(h)=h^{g}
$$
for all $ h \in G $ is easily seen as an automorphism of $ G $; is called the \emph{inner automorphism} induced by $g. $ An automorphism $ \alpha $ of $ G $ is called { $\operatorname{inner}$} if $ \alpha = \rho_ {g} $ for some $ g \in G $, and the set of all inner automorphisms they are denoted by $ \operatorname {Inn} (G)$. Obviously, { a group is abelian} if and only if $ \operatorname {Inn} (G) = 1. $ So if $G$ is abelian and $g,h\in G$ then $$\rho_g(h )=h^g=g^{-1}hg=g^{-1}gh=h.$$

An $N$ subgroup of $G$ that is invariant under all inner automorphisms (for which therefore $N ^ {g} = N$ for all $g \in G$) is called a { normal subgroup } of $ G. $ If $ N $ is a normal subgroup of $ G $, we denote it symbolically by $ N \unlhd G$ ( and by $ N \triangleleft G $ when $ N \neq G $).\medskip

We can observe that $1$ and $G$ are always normal subgroups of $G$, and a group $G \neq 1$ with no other normal subgroups is called a \emph{simple}. Thus $G$ is simple if and only if it has precisely two normal subgroups.\medskip

A subgroup $U$ of a group $G$ is considered subnormal in $G$ if there exists a chain of subgroups $U_{0},U_{1},\ldots,U_{r}$ of $G$ such that
$$
U = U_ {0} \unlhd U_ {1} \unlhd \cdots \unlhd U_ {r-1} \unlhd U_ {r} = G
$$
This is called a subnormal string from $U$ to $G$. We will call it a composition series if each of its \emph{factors} $$ U_{t} / U_{i-1}\ (i = 1, \ldots, r) $$ \emph{is simple}, in which case the factors are called \emph{composition factors} of $ G $.
%If $ U $ is subnormal in $ G $.\medskip \

%Es bien conocido, bastar \'ia revisar libros de \'algebra abstracta, que un grupo  $(G,*)$ es un conjunto $G$ dotado de una operaci\'on $*$ que es asociativa, posee elemento de identidad $e$ y es invertible con respecto a $e$.  Si la operaci\'on $*$ es conmutativa el grupo se dice abeliano.

%La clasificaci\'on de los grupos simples ha sido muy estudiada debido a su importancia, esta importancia, en cierto modo se debe al estudio de los grupos cuyos subgrupos simples son abelianos, es decir, de orden primo. Recuerde que por grupo simple es aquel cuyos \'unicos subgrupos abelianos son los triviales. \medskip 

%De esta clasificaci\'on se desprende el estudio de los grupos finitos resolubles, pues estos han sido objetos de gran atenci\'on, obteniendo as \'i muchos resultados  interesantes en los \'ultimos a\~nos los cuales se cumplen para estos tipos de grupos pero no en general.     \medskip 

\subsection{Group action}
%En un grupo $ G $ su operaci\'on binara es una funci\'on $G\times G\rightarrow G$ que asigna cada par $\left ( \sigma ,\tau  \right )$ de elementos de $G$ unico  $\sigma  \tau  \in G$.

We say that given a group $G$ and $X$ a non-empty set, an action of $G$ on $X$ is a function $$ G\times X\longrightarrow X$$ that assigns each element $\sigma \in G$ and $x \in X$ a single element $\sigma x \in X$. \medskip
The action of $G$ on $X$ checks the following properties:

\begin{enumerate}
    \item If $e$ is the identity element of $G$, then $\forall x\in X$ then $e x=x$.
    \item  If  $\sigma , \tau \in G$, then $\left ( \sigma  \tau  \right ) x=\sigma \left ( \tau  x \right )$.
\end{enumerate}
\begin{definition}
If $X$ is a set satisfying a action of $G$, then we say that $X$ is a $G$-Set.
\end{definition}  
Here are some examples of group actions:
\begin{example}
Consider any non-empty set $X$, let $ G\times X\longrightarrow X$ be defined with $ \sigma x= x $ . It is evident that this is an action of $G$ on $X$, in this case $G$ is said to act trivially on $X$.
\end{example}
\begin{example}
If $X$ is any non-empty set, let $G=S_{X}$ be the group of permutations of $X$. Then $G$ acts on $Y$ permuting its elements, this means that $ G\times X\longrightarrow X$ is given by $\ \sigma \ x=\sigma \left ( x \right ) $, where $\sigma \left ( x \right )$ is the function that calculates $x$.
\end{example}

\begin{example}
Let $G$ be a group and $H\leq G $. Then $H$ acts on $G$ through the product of $G$, that is, $ H\times G\longrightarrow G$ is given by $(h,\sigma)\mapsto h\sigma $(The action $\sigma$ gives us a translation of $h$ by the left). In particular, if $G=H$ then the group $G$ acts on itself.
\end{example}

\subsection{Semidirect product}

For this section we follow the references \cite{milne1996group, dummit1991abstract, mac1999algebra}. We start considering $N\triangleleft G$, i.e., $N$ is a normal subgroup of $G$, For each element $g$ of $G$ we can define an automorphism of $N$, $\varphi_g:\,N\rightarrow N$, such that $n \mapsto g n g^{-1}$. We can see that $\varphi$ induces the following homomorphism:
$$
\theta: G \rightarrow \operatorname{Aut}(N), \quad g \mapsto \varphi_{g} \mid N.
$$
On the other hand, if there exists a subgroup $Q$ of $G$ such that $G \rightarrow G / N$ induces the isomophism $Q \rightarrow G / N$, then  we can recover $G$ from $N, Q$, together with the restriction of $\theta$ to $Q$. Moreover, an element $g$ of $G$ can be written exclusively in the form
$$
g=n q, \quad n \in N, \quad q \in Q
$$
 $q$ should be the only one element of $Q$ mapping to $g N \in G / N$, and $n$ must be $g q^{-1}$. Thus, we have a one-to-one correspondence of sets
$$
%G \stackrel{1-1}{\longleftrightarrow} N \times Q .
G\rightarrow N\times Q, \quad \quad \quad 
N\times Q \rightarrow G
$$

If $g=n q$ and $g^{\prime}=n^{\prime} q^{\prime}$, then
$$
g g^{\prime}=(n q)\left(n^{\prime} q^{\prime}\right)=n\left(q n^{\prime} q^{-1}\right) q q^{\prime}=n \cdot \theta(q)\left(n^{\prime}\right) \cdot q q^{\prime}
$$

\begin{definition}\label{defsmp}
Let $N\triangleleft G$ and $Q\leq G$.
A group $G$ is a semidirect product of its subgroups $N$ and $Q$ whether the homomorphism $G \rightarrow G / N$ induces an isomorphism $Q \rightarrow G / N$.
\end{definition}

In an equivalent way, $G$ is a semidirect product of the subgroups $N$ and $Q$ whether
$$
N \triangleleft G ; \quad N Q=G ; \quad N \cap Q=\{1\}
$$
We can notice that $Q$ do not need to be a normal subgroup of $G$. When $G$ is the semidirect product of subgroups $N$ and $Q$, we write $G=N \rtimes Q$. That is, $N \rtimes_{\theta} Q$, where $\theta: Q \rightarrow \operatorname{Aut}(N)$ gives the action of $Q$ on $N$ by inner automorphisms.\\

\section{Solvable and nilpotent groups}
Solvable groups are very important in the abstract group theory as well in applications. In particular, solvability of groups give us some knowledge about mathematical physics concerning the phenomena, see \cite{abmw2011,athesis2009,manz1993,nicks2012}. In this section we present some aspects related to solvable and nilpotent groups, including properties, definitions and some known results. The concept of Frattini subgroup is also introduced.
\subsection{Nilpotent groups}
%Now, we provide some special types of groups.
	
Following \cite{dummit1991abstract} for any (finite or infinite) group $ G $ defines the following subgroups inductively:
$$ Z_ {0} (G) = 1, \quad Z_ {1} (G) = {Z}(G):=\{g \in G: \forall h \in G,\,\, g h=hg\}$$
and $Z_{i + 1}(G)$ is the subgroup of $G$ containing $Z_{i}(G)$ such that
$$
Z_ {i + 1} (G) / Z_ {i} (G) = Z \left(G / Z_ {i} (G) \right)
$$
(ie $Z_{i + 1}(G)$ is the entire preimage in $G$ of the center of $G/Z_{i}(G)$ under the natural projection). The chain of subgroups
$$
Z_{0} (G) \leq Z_{1} (G) \leq Z_{2} (G) \leq \cdots
$$
is called the upper central series of $G$. (The use of the term \emph{upper} indicates that $ Z_{i} (G) \leq Z_{i + 1} (G) $.) \medskip

\begin{definition}[{Nilpotent Group}]
A group $G$ is called {nilpotent} if it has a finite chain (or series) of subgroups $\left\{G_{i}\right\}_{i=1}^{n} \subset G$:
$$
\left\{1_{G}\right\}=G_{0} \subseteq G_{1} \subseteq \cdots \subseteq G_{n}=G
$$
where for each $i=0,1, \ldots, n-1$ it is true that:

\begin{description}
\item[$\bullet$] $G_{i}$ is a normal subgroup on $G_{i+1}$, denoted norms like $G_{i} \triangleleft G_{i+1}$ \item[$\bullet $] $G_{i+1} / G_{i} \leq Z\left(G / G_{i}\right)$ for $i=0,1, \ldots, n-1$.
\end{description}
\end{definition}

\begin{remark}
A group $ G $ is nilpotent if $ Z_ {c} (G) = G $ for some $ c \in \mathbb {Z} $. The smallest of such $c$ is called the nilpotence class of $G$. If $ G $ is abelian, then $ G $ is nilpotent, even more so of class 1, provided that $ | G |> 1$, since in this case $G = Z(G) = Z_{1}(G)$.
\end{remark} 

The following are some properties of nilpotent groups
\begin{enumerate}
\item Let $G$ be a nilpotent group of class $r$ and $H \leqslant G$ and $N \triangleleft G$. Then $H$ and $G / N$ are nilpotent.
    \item Every finite $p$-group is nilpotent.
    \item If $G$ and $G_1$ are nilpotent groups then $G \times G_1$ is nilpotent.
    \item  Let $G$ be a finite group. Then the following are equivalent.
    \begin{enumerate}
        \item  $G$ is nilpotent.
\item If $H<G$ then $H \neq N_{G}(H)$.
\item All maximal subgroups of $G$ are normal.
\item All Sylow subgroups of $G$ are normal.
\item $a b=b a$ for all $a, b \in G$ whose orders are coprime.
\item $G$ is the direct product of its Sylows subgroups.
    \end{enumerate}
    
\end{enumerate}

 In the study of maximal subgroups, G. Frattini in \cite{frattini1885intorno} introduces what is known today as the Frattini subgroup, which is the intersection of the maximal subgroups of the group $G$. That is to say 

$$\Phi(G)=\bigcap\{M \leq G \mid M \ \text{maximal in } \ \ G\}.$$

An important result, which has a simple proof, is the following. It gives a good argument to work with the normality of maximal subgroups and today we know it as Frattini's argument. 

\begin{theorem}
If $G$ is finite then $\Phi(G)$ is nilpotent.
\end{theorem}

This property can be used to characterize the nilpotency of a group $G$. If $G$ is a finite group then the following are equivalent:
 \begin{enumerate}
     \item $G$ is nilpotent.
  \item  $G / \Phi(G)$ is nilpotent.
  \item  Every maximal subgroup of $G$ is normal in $G$.
  \item  $G^{\prime} \leq \Phi(G)$
 \end{enumerate}
We recall that $G'$ denotes the derived of $G$, i.e., $G'=[G,G]$.
\subsection{Solvable groups}	
We start introducing the definition of Solvable group, also known as Soluble group, see \cite{milne1996group,dummit1991abstract,charris2,h2013I}.	

\begin{definition}[{Solvable group}]
A group $G$ is said to be \emph{solvable} if there is a finite chain of subgroups $\left\{G_{i}\right\}_{i=1}^{n} \subset G$:
$$
\left\{1_{G}\right\}=G_{0} \subseteq G_{1} \subseteq \cdots \subseteq G_{n}=G
$$
where for each $i=0,1, \ldots, n-1$ it is follows that:
\begin{enumerate}
\item [$\bullet$]$G_{i}$ is a normal subgroup on $G_{i+1}$, %usually noted as $G_{i} \triangleleft G_{i+1}$.

\item [$\bullet$] The quotient group $G_{i+1} / G_{i}$ is abelian.
\end{enumerate}

\end{definition}

Nilpotent groups are solvable and abelian groups are nilpotent. Unlike a solvable group, where the quotients of the normal series are abelian.\\

The following are some properties of solvable groups.

\begin{enumerate}
    \item  Every subgroup and every quotient of a solvable group is solvable.
    \item If $N \triangleleft G$ is such that $N$ is solvable and $G / N$ is solvable, then $G$ is solvable.
    \item Every finite $p$-group is solvable.
\end{enumerate}

Moreover, for all group $G$ satisfying $\Phi(G)=1$, all maximal subgroups of $G$ are solvable groups. Let $N$ be a nontrivial normal subgroup of $G$. Then, there exists a maximal subgroup $M$ of $G$ satisfying $G=NM$. Thus, if $N$ and $M$ are solvable, then $G$ is solvable too. It is consequence of Lemma \ref{Lemma 2.2}.

\section{Supersolvable Groups and Zarrin Criteria for Solvable Groups}

In this section we provide the concepts and known results about Supersolvable groups, also known as Supersuble groups or Polycyclic groups. Following \cite{r2012} and \cite{whvw1982} we have the following supersolvable group definition. Moreover, we present some results of Zarrin about non-nilpotent groups and their relation with Simple groups.

 %{Grupos superresolubles}	
		\begin{definition}[{Supersolvable group}]
		A group $G$ is {supersolvable} if there exist normal subgroups $G_{i}$ with
$$
\left\{1_{G}\right\}=G_{0} \subseteq G_{1} \subseteq \cdots \subseteq G_{n}=G
$$
and where each factor $ G_ {i} / G_ {i-1} $ is cyclic for $ 1 \leq i \leq n $. It is clear that groups whose order is a power of a prime $p$ are supersolvable.	
	\end{definition}

%Simple, non-abelian groups whose subgroups are solvable are what we call \emph{minimal simple groups}.

%\section{Algunos preliminares y ejemplos de grupos superresolubles}
 %{Algunos preliminares y ejemplos}

\begin{example}
The group $S_3$ is supersolvable. It is due to $1_G\leq A_3\leq S_3$
\end{example}

\begin{example}
The group $A_{4}$ is not a supersolvable group, it is because $K_{4}$ is the unique normal subgroup of $A_{4}$.
\end{example}

% \subsection{Algunos preliminares y ejemplos}

 Finitely generated abelian groups are supersolvable. It is because if $G$ es finitely generated (cyclic) we can consider $$G=\langle g_1,g_2,\cdots, g_n\rangle$$ which is an abelian group. We can set $G_i=\langle g_i,g_{i+1},\cdots, g_n\rangle$, and we can notice  $G_i/G_{i-1}=\langle g_iG_{i-1}\rangle$ y $G_{i} \unlhd G, \forall 1 \leq i \leq n$, we have a normal chain of subgroups 
 $$1=G_1\subseteq G_2\subseteq\cdots\subseteq G_n=G$$ where the factors are cyclics. For instance, $G$ is supersolvable.\medskip
 
\begin{remark}
We can observe that supersolvable groups are finitely generated.
\end{remark}

It can be seen that supersolvable groups are solvable since the factors $G_i/G_{i-1}$, are abelian (cyclic $\Rightarrow$ abelian) but being solvable does not imply being supersolvable (abelian is not necessarily cyclic).\medskip

\begin{example} 
We observe that $S_{4}$ is solvable because
$$
1 \subseteq K_{4} \subseteq A_{4} \subseteq S_{4}
$$
but $S_4$ has only the normal subgroups $A_{4}$ and $K_{4}$, therefore $S_{4}$ is not supersolvable
\end{example}

 %{Algunos preliminares y ejemplos}

\begin{theorem}[See \cite{doerk2011,whvw1982}]
The following statements hold.
\begin{enumerate}
\item[ (a)] Every subgroup of a supersolvable group is supersolvable.
\item[ (b)] The image of homomorphims of a supersolvable group is supersolvable.
%\item[ (c)] El producto directo de dos grupos superresolubles es superresoluble.
%\item[ (d)] Si el grupo $ G $ tiene dos subgrupos normales $ H $ y $ K $ con $ G / H $ y $ G / K $ superresolubles, entonces $ G / H \cap K $ es superresoluble.
\end{enumerate}
\end{theorem}

To illustrate the reader, we write the proof of the previous theorem for completeness. It is because it is classical result. For further details see \cite{whvw1982}.
\begin{proof}
We proceed according to each item.\medskip

 (a) Let $  {G} $ be a supersolvable group and let 
$$\{1 \} = G_ {0} \unlhd G_ {1} \unlhd G_ {2} \unlhd \ldots \unlhd G_ {r} = G$$
be a principal serie of $ G $. For the subgroup $ S $ of $ G $,
$$\{1 \} = S \cap G_ {0} \unlhd S \cap G_ {1} \unlhd S \cap G_ {2} \unlhd \ldots \unlhd S \cap G_ {r} = S $$

is a normal serie of $ S $ in where the factor $ \left(S \cap G_ {i} \right) / \left(S \cap G_ {i-1} \right) $ is isomorphic to the subgroup $\left(S \cap G_{i}\right) G_{i-1} / G_{i-1}$ de $G_{i} / G_{i-1}$.\medskip

(b) For all $ i, \ G_ {i} / G_ {i-1} $ is of first order, such that $ \left(S \cap G_ {i} \right) / \left(S \cap G_ {i-1 } \right) $ is either trivial or of first order. Thus, the previous normal serie for $ S $ will produce, after the elimination process of redundant terms, a principal serie of $ S $ with cyclic factors. Therefore, $ S $ is supersolvable. %Si $ \alpha $ es un homomorfismo de $ G $ en el grupo $ T $, entonces
\end{proof}

 %{Algunos preliminares y ejemplos}

\begin{theorem}[See \cite{doerk2011,whvw1982}]

Assume $ H \leq G, N \unlhd G $, where $ G $ is a supersolvable group. Then $ H $ and $ G / N $ are supersolvable groups.

\end{theorem}

\begin{theorem}[See \cite{doerk2011,whvw1982}]
The following statements hold.
\begin{enumerate}
\item[(1)] A direct product of a finite number of supersolvable groups is supersolvable.
\item[(2)] If $ H_ {1},  \ldots, H_ {n} $ are normal subgroups of $ G $ and the groups $ G / H_ {1}, \ldots, G / H_ {n} $ are supersolvable, then
$ G / \bigcap_ {i = 1} ^ {n} H_ {i} $ is supersolvable.
\end{enumerate}
\end{theorem}
We can summarize previous results as follows:
$$ cyclic  \subset   abelian \subset  nilpotent  \subset  solvable \subset  all\,\,groups$$
$$ cyclic  \subset   abelian \subset nilpotent \subset  supersolvable  \subset  solvable \subset  all\,\,groups$$

We recall that under special conditions some nilpotent groups are supersolvable groups. \\

 \textbf{Remark.} Given a class of groups $\mathfrak{X}$, the group $G$ is said to be a non-minimal group on $\mathfrak{X}$, or a critical group on $\mathfrak{X}$, if $G \notin \mathfrak{X}$ but all proper subgroups of $G$ are in $\mathfrak{X}$. The detailed knowledge of the structure of non-minimal groups in $\mathfrak{X}$ can provide insight into what makes a group a member of $\mathfrak{X}.$ All groups considered in this document are finite. See \cite{doerk2011,h2013I}.\\
 
We recall the importance work of Zarrin because he generalized the result of Schmidt for solvability of groups from its non-nilpotent subgroups.  Such result motivates a new criteria for solvability of groups through the number of their non-supersolvable and supersolvable subgroups. Zarrin introduced the class $\mathcal{S}^n$, that is, a group $G$ belongs to $\mathcal{S}^n$ if it contains exactly  $n$ non-nilpotent subgroups. Thus, Zarrin proved the following results in \cite{zarrin2012}.

%Según el Teorema de Schmidt, un grupo finito cuyos subgrupos propios son todos nilpotentes (o un grupo finito sin subgrupos propios no nilpotentes) es soluble. En este trabajo demostramos que todo grupo finito con menos de 22 subgrupos no nilpotentes es soluble y que esta estimación es aguda.}

\begin{theorem}
Let $G$ be a finite group that is an $\mathcal{S}^{n}$-group with $n\leq22$. Then
\begin{enumerate}
 \item $G$ is solvable,
\item $A_5$ is an $\mathcal{S}^{22}$-group. 
\end{enumerate}
\end{theorem}

\begin{theorem}
Let $G$ be a semi-simple non-abelian finite group such that it is an $\mathcal{S}^n$-group with $n\leq65$. Then $G$ is isomorphic to $A_5, \operatorname{Sym}(5)$, or $\operatorname{SL}_2(7).$
\end{theorem}

%section{Motivation}

We must take into account the following properties that satisfy the $\mathcal{S}^n$-groups introduced by Zarrin:\\

\begin{enumerate}
    \item Consider $G=H\times K\in \mathcal{S}^m$ and $K\in S^m$. Then $G\in \mathcal{S}^r$ with $mn\leq r.$

 \item Given a group $G$, and $K$ a non-nilpotent normal subgroup of $G$ such that $G/K\in S^n$ and $K\in \mathcal{S}^m$. Then $G\in \mathcal{S}^t$ for $t\geq m+n.$

 \item Let $\operatorname{SL}_2(q)\in \mathcal{S}^m, \ m\in \mathbb{N}$ and $q\geq 4$ be a power of $p$ prime. Then $\operatorname{SL}_2(q)$ has

\begin{equation*}
\begin{array}{ll}
q^{2}+1 & \text { if } q=2^{s} \\
\frac{q^{2}-q+2}{2} & \text { if } q=2^{f}+1 \\
\frac{q^{2}+q+2}{2} & \text { if } q=2^{f}-1 \\
q^{2}+1 & \text { if neither } q=2^{\mathrm{f}}+1 \text { nor } \mathrm{q}=2^{\mathrm{f}}-1
\end{array}
\end{equation*}
non-nilpotent subgroups ($\leq m$).
\item If $\operatorname{Sz}(q)$ with $q=2^{2m+1}, \ m>0$. Then has $\frac{q^4+q^2+2}{2}\leq n$ non-nilpotent subgroups.\\
\end{enumerate}

To see these results, note that the groups given in Thompson's classification have the subgroups described by Dickson.\\

We use the standard notation followed in Doerk and Awkes \cite{doerk2011} or Huppert \cite{h2013I}. We use $\operatorname{SL}_{m}(q)$ and $\operatorname{PSL}_{m}(q)$ to denote the special linear group and the special projective linear group, respectively, of dimension $m$ over the field with $q$ elements, where $q$ is a power of a prime. It can be seen that
$$
\left|\operatorname{PSL}_{2}\left(p^{f}\right)\right|={p^{f}\left(p^{f}-1\right)\left(p^{f}+1\right) \over \left(2, p^{f}-1\right)}
$$

{\bf The Suzuki group}. The group $\operatorname{Sz}(\mathbb{F}_q)=\operatorname{Sz}(q)$ is defined as the set of linear maps over the vector space $V$ which preserve the inner product and also the restricted outer product. It can be shown that $\operatorname{Sz}(q)$ acts 2-transitively on the set of points $q^{2}+1$, and that the two-point stabilizer has order $q-1$. Therefore
$$
|\operatorname{Sz}(q)|=(q^{2}+1) q^{2}\left(q-1\right) .
$$

We recall that a minimal simple group is a simple group whose maximal subgroups are solvable. In this way, Thompson in \cite{thompson1973nonsolvable} classified the minimal simple groups starting from the work with Walter Feit \cite{feit1963solvability}. As Thompson explained, this classification is a translation from solvable groups to simple groups. This allows us to know in a deep way the properties of solvable groups in the study of structures and solvable subgroups of a simple group. Briefly, Thompson opted for a classification of minimal simple groups for small orders. 
%%%%%%%%%%%%%%%%%%%%%%%%%%%%%%

 \begin{theorem}[{Thompson \cite{thompson1973nonsolvable}}]
   A minimal simple groups are a the following groups: 
 
 \begin{enumerate}
\item $ \operatorname{P S L}_2 ( p) $, where $ p $ is a prime number with $ p> 3 $ y $ p ^ {2} -1 \equiv 0 \ (5) $.
\item $ \operatorname{P S L}_2 \left (2 ^ {q} \right) $, where $ q $ is a prime number.
\item $ \operatorname{P S L }_2\left(3 ^ {q} \right) $, where $ q $ is an odd prime number.
\item $ \operatorname{P S L}_3 (3) $.
\item The Suzuki groups $\operatorname{Sz}\left (2 ^ {q} \right)$, where $ q $ is an odd prime number. %(Trataremos de estos grupos en detalle en el Volumen II).
 
 \end{enumerate}
 \end{theorem}

In \cite{dickson1959linear} Dickson et al investigated exhaustively the classical finite groups through linear groups. In particular, in \cite[\S 12]{dickson1959linear} the authors listed the subgroups of $\operatorname{PSL}_2(p^f)$. This work of Dickson et. al.  was based on the previous works of Moore and Wiman. For more details see \cite{dickson1899report}.
 
\begin{theorem}
[Dickson  \cite{h2013I,dickson1959linear,dickson1899report} ] The group $\operatorname{P S L}_2\left( p^{f}\right)$ has only the following subgroups
	\begin{enumerate}
		\item  Elementary $p$-abelian groups
\item  Cyclic Groups of order $z$ with $z | \frac{p^{f} \pm 1}{k}$, where $k=\operatorname{mcd}\left(p^{f}-1,2\right)$.
\item  Dihedral groups of order $2 z$ with $z$ such as in $(2)$.
	\item  Alternating Group $A_{4}$ for $p>2$ or $p=2$ and $f \equiv 0\ (2)$.
	\item  Symmetric Group $S_{4}$ for $p^{2 f}-1 \equiv 0\ (16)$.
\item  Alternating Group $A_{5}$ for $p=5$ or $p^{2 f}-1 \equiv 0 \ (5)$.
\item  Semidirect products of elementary abelian groups with order $p^{m}$ with cyclic groups of order $t$; where $t \mid p^{m}-1$ and $t |p^{f}-1$
\item  $\operatorname{P SL}_2\left( p^{m}\right)$ for $m|f$ and $\operatorname{P G L}_2\left( p^{m}\right)$ for $2 m | f$.	
\end{enumerate}

\end{theorem}

Based on the results of Zarrin in \cite{zarrin2012}, Zarrin himself obtained that the only simple groups with exactly 22 non-nilpotent subgroups is $A_5$ and that the only simple group with exactly 65 non-nilpotent subgroups is $\operatorname{SL}_2 (7)$. A corollary to these facts is that if a non-solvable group $G$ is an $S^{n}$-group with $n\leq 22$, then $G/\Phi(G)$ is isomorphic to $A_5$. Zarrin in \cite{zarrin2012} conjectured that a finite non-solvable group that is a $S^{22}$-group is isomorphic to $A_5$ or $\operatorname{SL}_2(7)$. This conjecture was proven by Ballester-Bolinches, Esteban-Romero and Lu in 2017, see \cite[Theorem A]{BRL2017} for further details. The following theorem corresponds to the proof of the mentioned conjecture.

\begin{theorem}[See \cite{BRL2017}]
 Let $G$ be an insoluble group. Then $G$ has exactly 22 non-nilpotent subgroups if and only if it is isomorphic to $A_{5}$ or $\operatorname{SL}_{2}(5)$.
\end{theorem}

Earlier, Huppert \cite{huppert1954} proved that the word nilpotent in Schmidt's theorem can be replaced by supersolvable obtaining the same conclusion mutatis mutandis. 

%\textcolor{red}{Revisar que todo el documento tenga siempre solvable, supersolvable y supersolvability}

\subsection{Groups with non-supersolvable subgroups}
One natural question that arise is about the minimum number of non-supersolvable subgroups to guarantee the solvability of a group $G$. The results of the paper \cite{BRL2017} were motivated by the paper \cite{zarrin2012} where they shown that { the only minimal simple group with exactly 6 non-supersolvable subgroups is $A_5$ and that an unsolvable group $G$ has exactly 6 non-supersolvable subgroups if and only if it is isomorphic to $A_{5}$ or $\operatorname{SL}_{2}(5)$}. The following results can be found in \cite{BRL2017}.

	\begin{lemma}\label{Lemma 2.1}
		Let $G$ be a group. The number of non supersolvable subgroups of $G / \Phi(G)$ is not bigger than the number of supersolvables subgroups of $G$. \medskip \
	\end{lemma}

In the following lemma it is related $G/\Phi(G)$ with minimal simple groups. Remember that $\Phi(G)$ is the Frattini subgroup of $G$.

	\begin{lemma}\label{Lemma 2.2}
		Let $G$ be an non solvable group whose maximal subgroups are solvable. Then $G / \Phi(G)$ is a minimal simple group.
	\end{lemma}
	
A consequence of the simple finite groups classification, according to Thompson, is that all non-abelian simple group contains a minimal simple subgroup.
	
	\begin{lemma}\label{Lemma 2.3}
The number of non-supersolvable subgroups of a minimal simple group is at least $ 6. $ The only minimal simple group with exactly 6 non-solvable subgroups is $ A_ {5} $.
	\end{lemma}

The proof of this lemma is obtained through Thompson classification on simple minimal groups together with Dickson classification. We proceed by making a discrimination of the subgroups of each minimal simple group to which $G$ can be isomorphic of the Thompson classification. See \cite[Kapitel II, Bemerkung 7.5]{h2013I} and \cite[Kapitel II, Hilfssatz 6.2]{h2013I}. Detailed proofs of these lemmas can be found in \cite{BRL2017}. In this way, with these lemmas Ballester-Bolinches, Esteban-Romero and Lu in \cite{BRL2017} proven the following theorem.

	\begin{theorem}[Theorems B and C in \cite{BRL2017}]\label{teoBC}
	The following statements hold.
		\begin{enumerate}
			\item A group with less than $6$ non-supersolvable subgroups is solvable.
			\item Let $G$ be an unsolvable group. Then $G$ has exactly 6 non-supersolvable subgroups if and only if it is isomorphic to $A_{5}$ or $S L_{2}(5)$.
		\end{enumerate}
	\end{theorem}
	
 For this proof a minimal counterexample is constructed. See the proof in \cite{BRL2017}.  To prove (1) the authors used the method of reductio ad absurdum considering false the result and assuming $G$ as the non-solvable group of smaller order with less than 6 non-supersolvable subgroups. The contradiction arrives by the Lemma \ref{Lemma 2.1} and Lemma \ref{Lemma 2.2}  because a contradiction of Lemma \ref{Lemma 2.3} is obtained. To prove (2) the authors used a similar argument and also they used (1).

\subsection{Groups with supersolvable subgroups}
Motivated for \cite{BRL2017} the authors Jiakuan Lu and Jingjing Wang in \cite{lu2020finite} proved the solvability of a finite group with fewer than 53 supersolvable subgroups. Furthermore, they showed that a finite non-solvable group has exactly 53 supersolvable subgroups if and only if it is isomorphic to $ A_5$.\\

%Also, it is shown that a finite group with less than 7 conjugation classes of supersolvable subgroups is solvable, and the only finite non-solvable group with 7 conjugation classes of supersolvable subgroups is isomorphic to $A_5$.\\

\textbf{  A class of groups: } \ \ In \cite{lu2020finite} the class $ \mathcal{S}^{n} $ was introduced as; $ \mathcal{S}^{n}$ as the set of finite groups having exactly $n$ supersolvable subgroups, and a group $ G $ is said to be a $ \mathcal{S}^{n}$-group if $ G \in \mathcal{S}^n .$  \\  %Primero probamos el siguiente resultado.

 \begin{lemma}\label{Lemalu 2.1} Let $ G $ be a  $ \mathcal{S}^{n}$-group with the Frattini subgroup $\Phi (G)$, and let $ G / \Phi (G) $ be a $ \mathcal{S}^{m}$-group. Then $ m \leq n $.
\end{lemma}

We can observe that this lemma is equivalent to Lemma \ref{Lemma 2.1} in the reference  \cite{BRL2017}, which was the motivation of the recent paper \cite{lu2020finite} and contains a similar proof.

\begin{lemma}\label{Lemalu 2.2} Let be $ G = \operatorname {PSL}_2 (q) $, and let $ p $ be a prime number. Assume $q=p^{f} \geq 8 $ for $ p = 2 $, or $ q = p ^ {f} \geq 13 $ for $ p >2$ is prime. Then $ G $ is a $ \mathcal{S}^{n}$-group for $ n \geq 2 \left (q ^ {2} +1 \right) $.
\end{lemma}

To see this lemma it is important to review the conjugation classes of simple groups, i.e., observing the Thompson classification and the quantity of supersolvable groups of $ G = \operatorname {PSL}_2 (q) $. In particular, observing the number of dihedral subgroups, we conclude that they are supersolvable and they contain two conjugation classes: type $D_{2(q-1)/d}$ and type  $D_{2(q+1)/d}$ respectively, where $d=(2,q-1).$

\begin{lemma}\label{Lemalu 2.3} Let $ G $ be a non-solvable group whose maximal subgroups are solvable. Then $ G / \Phi (G) $ is a minimal simple group.
\end{lemma}

\begin{lemma}\label{Lemalu 2.4} A minimal simple group is a $\mathcal{S}_{n}$-group with $ n \geq 53. $ The only minimal simple group in $ \mathcal{S}^{53} $ is $ A_{ 5} $.
\end{lemma}

In a similar way as Lemma \ref{Lemma 2.3}, these lemmas are obtained through Thompson classification on simple minimal groups together with Dickson classification.  Further, we provide examples to compute the number of supersolvable and non-supersolvable subgroups of $A_5$ as well their implementation in GAP system.

\begin{theorem} Let $ G $ be a  $ \mathcal{S}^{n} $-group.
\begin{enumerate}
\item[(1)] If $ n $ smaller than $53$, then $G$ is solvable.
\item[(2)] If $ G $ is a non-solvable $ \mathcal{S}^{n} $-group, then $ n = 53 $ if and only if $ G \cong A_ {5} $.
\end{enumerate}
\end{theorem}

The proof of previous theorem is similar to the proof of Theorem \ref{teoBC}. For this proof a minimal counterexample is constructed. See the proof in \cite{lu2020finite}. To prove (1) the authors assumed false the result (reductio ad absurdum) and also they considered  $G\in\mathcal{S}^{n}$ with $n<53$ as the nonsolvable group with smaller order. Using Lemma  \ref{Lemma 2.2} and Lemma \ref{Lemalu 2.2} they arrived to a contradiction of Lemma \ref{Lemalu 2.4}. To prove (2) the authors used similar statements and item (1).

\section{New contributions}
In this section we present some algorithms in GAP related to the previous results, in particular, we present some programming modules about supersolvable, non-supersolvable and non-nilpotent subgroups of the alternating group of five elements $A_5$. Finally, we provide some results concerning the relation between (p,q)-groups and dihedral groups with supersolvability. .

\subsection{Programs in GAP System}
Now we observe how to compute the quantity of subgroups of certain type simple groups, providing by Thompson classification, using Dickson classification for subgroups. Due to $A_5$ is the group with smaller order in the classification of simple groups, see ATLAS \cite{conway1985} for description of subgroups of $A_5$, we focus on $A_5$ to illustrate these programs.\\

Recall that $\operatorname{PSL}_2(4)\cong A_5\cong \operatorname{PSL}_2(5)$ with $|\operatorname{PSL}_2(5)|=60$ and we know the description of the subgroups of  $\operatorname{PSL}_2(p)$ because result of Dickson. Thus, $\operatorname{PSL}_2(5) $ has the following subgroups:\\

\begin{enumerate}
\item[{ Type (1) }] elementary abelian $p$-groups (2-Sylow, 3-Sylow y 5-Sylow),
\item[{ Type (2) }] cyclic groups ($C_2, C_3$),
\item[{ Type (3) }] dihedral groups $D_6\cong S_3$ y $D_{10}$,
\item[{ Type (4) }] $A_4$,
\item[{ Type (6) }] $A_5$,
\item[{ Type (7) }] Semidirect products of elementary abelian groupos of order $p^{m}$ with cyclic groups of order $t$; where $t \mid p^{m}-1$ and $t |p^{f}-1$; ($C_2\times C_2$)\\
\end{enumerate}

We recall that the number of conjugated subgroups is given by the index of the normalizer  

$$[G:N_G(H)]=|\{H^g|g\in G\}|$$
 
 and that if $P$ is a Sylow $p$-subgroup, then $$[G:P]=[G:N_G(P)][N_G(P):P] \quad \text{or} \quad m=n_p[N_G(P):P]$$
where $m$ is the index of Sylow subgroup. Thus, $n_p\equiv 1 \mod p$. \\

Also we recall that $G=\operatorname{PSL}_2(5)$ is simple (it has no non-trivial normal proper subgroups). Thus, for all subgroup $H$ it is satisfied $N_G(H)=H$, because the normalizer is a normal subgroup. In this way, we can compute the number of conjugated subgroups to $D_6$ and $D_{10}$ as $$[G:D_{10}]=\frac{60}{10}=6, \quad [G:D_{6}]=\frac{60}{6}=10.$$
In a similar way, of Type $A_4$ there $[G:A_4]={60}/{12}=5$ conjugated subgroups to $A_4$.\\

Similarly, we can apply the same methodology to each type of and observing the number of Sylow subgroups of each Sylow subgroup of $G$. Thus, we can obtain that the number of subgroups of $A_5$ is  59. Thus, the non-supersolvable subgroups of $A_5$ are $A_4$ and $A_5$, being 5 of type $A_4$ and 1 of type $A_5$. In consequence, $A_5$ has 6 supersolvable subgroups and 53 supersolvable subgroups. Thus, due to $A_5$ is minimal simple group with smaller order, we can give some conjectures about the quantity of certain kind of subgroups of a particular group $G$ in order that $G$ must be a solvable group. In the same way, we can obtain the number of non-nilpotent subgroups of $G$ using the Thompson classification and the Dickson classification of $A_5$.  \\

In this way, we present the following algorithm, implemented by the authors in GAP.

\begin{algorithm}
\caption*{\bf Computing the number of non-supersolvable subgroups of a group of $A_5$}
\begin{verbatim}
gap> G:=AlternatingGroup(5);
Alt( [ 1 .. 5 ] )
gap> L:=LatticeSubgroups(G);
<subgroup lattice of Alt( [ 1 .. 5 ] ), 9 classes, 59 subgroups>
gap> C:=ConjugacyClassesSubgroups(L);
[ Group( () )^G, Group( [ (2,3)(4,5) ] )^G, Group( [ (3,4,5) ] )^G, 
Group( [ (2,3)(4,5), (2,4)(3,5) ] )^G, Group( [ (1,2,3,4,5) ] )^G, 
Group( [ (1,2)(4,5), (3,4,5) ] )^G, Group( [ (1,4)(2,3), (1,3)(4,5) ] )^G,
Group( [ (3,4,5), (2,4)(3,5) ] )^G, 
Group( [ (2,4)(3,5), (1,2,5) ] )^G ]
gap>FNS:=Filtered (C, t -> not IsSupersolvable(Representative(t)));
[ Group( [ (3,4,5), (2,4)(3,5) ] )^G, Group( [ (2,4)(3,5), (1,2,5) ] )^G ]
gap> S:=Sum(List(FNS, Size));
6
\end{verbatim}
\end{algorithm}

\begin{algorithm}
\caption*{\bf Computing the number of supersolvable subgroups of a group of $A_5$}
\begin{verbatim}
gap> G:=AlternatingGroup(5);
Alt( [ 1 .. 5 ] )
gap> L:=LatticeSubgroups(G);
<subgroup lattice of Alt( [ 1 .. 5 ] ), 9 classes, 59 subgroups>
gap> C:=ConjugacyClassesSubgroups(L);
[ Group( () )^G, Group( [ (2,3)(4,5) ] )^G, Group( [ (3,4,5) ] )^G, 
Group( [ (2,3)(4,5), (2,4)(3,5) ] )^G, Group( [ (1,2,3,4,5) ] )^G, 
Group( [ (1,2)(4,5), (3,4,5) ] )^G, Group( [ (1,4)(2,3), (1,3)(4,5) ] )^G,
Group( [ (3,4,5), (2,4)(3,5) ] )^G, 
Group( [ (2,4)(3,5), (1,2,5) ] )^G ]
gap>FS:=Filtered (C, t ->  IsSupersolvable(Representative(t)));
[ Group( () )^G, Group( [ (2,3)(4,5) ] )^G, Group( [ (3,4,5) ] )^G, Group( [ (2,3)(4,5), (2,4)(3,5) ] )^G,
Group( [ (1,2,3,4,5) ] )^G, Group( [ (1,2)(4,5), (3,4,5) ] )^G, 
  Group( [ (1,4)(2,3), (1,3)(4,5) ] )^G ]
gap>SS:=Sum(List(FS, Size));
53
\end{verbatim}
\end{algorithm}

\begin{algorithm}
\caption*{\bf Computing the number of non-nilpotent subgroups of a group of $A_5$}
\begin{verbatim}
gap> G:=AlternatingGroup(5);
Alt( [ 1 .. 5 ] )
gap> L:=LatticeSubgroups(G);
<subgroup lattice of Alt( [ 1 .. 5 ] ), 9 classes, 59 subgroups>
gap> C:=ConjugacyClassesSubgroups(L);
[ Group( () )^G, Group( [ (2,3)(4,5) ] )^G, Group( [ (3,4,5) ] )^G, 
Group( [ (2,3)(4,5), (2,4)(3,5) ] )^G, Group( [ (1,2,3,4,5) ] )^G, 
Group( [ (1,2)(4,5), (3,4,5) ] )^G, Group( [ (1,4)(2,3), (1,3)(4,5) ] )^G, 
Group( [ (3,4,5), (2,4)(3,5) ] )^G, 
Group( [ (2,4)(3,5), (1,2,5) ] )^G ]
gap>FN:=Filtered (C, t -> not IsNilpotent(Representative(t)));
[ Group( [ (1,2)(4,5), (3,4,5) ] )^G, Group( [ (1,4)(2,3), (1,3)(4,5) ] )^G,  
Group( [ (3,4,5), (2,4)(3,5) ] )^G, Group( [ (2,4)(3,5), (1,2,5) ] )^G ]
gap>SN:=Sum(List(FN, Size));
22
\end{verbatim}
\end{algorithm}

\subsection{Supersolvability of $(n,m)$-groups and dihedral groups}

The following elementary result illustrates the concept of supersolvability.

\begin{proposition}\label{prop1:aclm}
The dihedral group $D_{n}=\left\langle r, s \mid r^{n}=1, s^{2}=1, s r s=r^{-1}\right\rangle
$ is supersolvable.
\end{proposition}

\begin{proof}
Due to the dihedral group contains only one normal proper subgroup, which corresponds to $R_n=\left\langle r\right\rangle$ and $D_n$ is semidirect product of $C_2$ with $R_n$, that is $D_n\cong C_2\rtimes R_n$, then the homomorphism  $\varphi: \,D_n\to D_n/R_n$ induces the isomorfism $\phi:\,C_2\to D_n/R_n$. In this way,  $1\leq R_n\leq D_n$ is a chain of normal subgroups and for instance $D_n/R_n$ is isomorphic to $C_2$ that is cyclic. Thus, we conclude $D_n$ is supersolvable. 
\end{proof}

The following example corresponds to the group of symmetries of an square, which is isomorphic to $D_4$.

\begin{example}
The group $G:=\langle(1,2,3,4),(1,3)\rangle \cong D_{4}$ is supersolvable, where $D_4$ is the well known dihedral group or order $8$. Considering $G_1$ and $G_2$ subgroups of $G$ such that $G_{1}=\langle(1,3)(2,4)\rangle$  and  $G_{2}=\langle(1,2,3,4)\rangle$, then  $$1_G=G_{0} \leq G_{1} \leq G_{2} \leq G_{3}=G$$
is a normal chain with cyclic factors. Thus, we can conclude that $D_4$ is supersolvable.
\end{example}

Although this is a review paper, we present this original result that is elementary, but it is new as far as we know.\\

Now groups of type $(n,m)$ introduced in \cite{charris2,acostateoremas, charris1} are defined. A group of type $(n,m)$ is a group of the following form $$G:=\left< s,r \ | \ s^n=r^m=1, \ sr=r^ts \ \text{with} \ n, m \in \mathbb{Z}_+, \ 1<t<m \right>.$$

Furthermore, each element $x$ of the group $G$ can be written uniquely as $$x=s^ir^j \quad \quad i,j\in\mathbb{Z}, \ \ 0\leq i < n, \ \ 0\leq j < m. $$

It can be seen that $m\geq 3$ and $G$ is not abelian, so $n\geq 2$ and from the definition it follows that $|G|=nm.$ We say that this group $G$ is of type $(n,m)$ and generated $(s,r)$. In addition, the following properties are deduced:

\begin{enumerate}
    \item [(i)] $sr^{t^k}=r^{t^{k+1}}s, \quad \quad k=0,1,2, \ldots$
    \item [(ii)] $s^nr=r^{t^k}a^k, \quad \quad k=0,1,2,\ldots$
    \item [(iii)] $s^kb^i=b^{it^k}a^k, \quad \quad n\leq0, \ i\in\mathbb{Z}$
\end{enumerate}

\begin{theorem}\label{semi}
Let $G$ be a group of type $(n,m)$ and generated by $(s,r),$ and if $H=\left<s\right>$ and $N=\left<r\right> $. So $G=H \rtimes N$ with $N$ normal in $G.$
\end{theorem}

\begin{theorem}
Let $G$ be a group of type $(n,m).$ Then $G$ is a supersolvable group.
\end{theorem}
\begin{proof}
It is an immediate consequence of the theorem \ref{semi}. We have that the homomorphism $G\to G/N$ induces the isomorphism $H\to G/N$. Also, both subgroups are cyclic, so we have
 $$\{1\} \subseteq N \subseteq G $$ with $G/N\cong H$. Thus, $G$ is supersolvable.
 \end{proof}
 
 \section*{Final Remarks}
In this paper, which falls in the intersection of abstract group theory, combinatorial group theory and computational group theory, we made a review about the solvable, supersolvable and nilpotent groups relative to the computation of subgroups. We studied the more important concepts, properties and theorems that allowed us the classification of groups with non-solvable, non-supersolvable and non-nilpotent subgroups. Finally, we gave as contribution the supersolvability analysis of  (n,m)-groups and dihedral groups.\\

In the last years the structure of groups through their subgroups have deeply studied, that is, the quantity of subgroups or certain class to determine the group. In particular, Zarrin proved that a group with more than 26 subgroups with normalizers are solvable groups. In addition, a group with more than 22 derived subgroups is a solvable. In other way, the only one minimal simple group with exactly 23 derived subgroups is $A_5$. 

Further research can involve more general aspects such as follows:

Given two classes of groups $X$ and $Y$, provide a characterization of groups with $n$ subgroups non-X-subgroups not being Y-groups. In previous cases, $X$ means nilpotent and $Y$ means solvable. The interested reader can change $X$ and $Y$ by another structures.
\subsection*{Acknowledgements}
The research of the first author has been as a result of the Mathematics Scientific Writing lectures, within the MESCyT 2016-2017-081 research project.
The second author and third author have been supported by the MESCyT research project 2016-2017-081.
The authors thank the referee for his valuable suggestions.

\bibliographystyle{unsrt}
\bibliography{FArxivALM}

\begin{thebibliography}{10}

\bibitem{galois1830}
{\'E}.~Galois.
\newblock Analyse d’un m{\'e}moire sur la r{\'e}solution alg{\'e}brique des
  {\'e}quations.
\newblock {\em Bulletin des sciences math{\'e}matiques physiques et chimiques},
  13(55):171--172, 1830.

\bibitem{galois1830note}
{\'E}.~Galois.
\newblock Note sur la r{\'e}solution des {\'e}quations num{\'e}riques.
\newblock {\em Bulletin des Sciences math{\'e}matiques XIII}, 413, 1830.

\bibitem{feit1963solvability}
W.~Feit and J.~G. Thompson.
\newblock {\em Solvability of groups of odd order}.
\newblock Pacific Journal of Mathematics, 1963.

\bibitem{weinstein1982}
M.~Weinstein, H.~G. Bray, W.~E. Deskins, D.~Johnson, J.~F. Humphreys, B.~M.
  Puttaswamaiah, P.~Venzke, and G.~L. Walls.
\newblock Between nilpotent and solvable.
\newblock {\em New Jersey}, 1982.

\bibitem{baer1955}
R.~Baer.
\newblock Supersoluble groups.
\newblock {\em Proceedings of the American Mathematical Society}, 6(1):16--32,
  1955.

\bibitem{remeslennikov1969}
V.~N. Remeslennikov.
\newblock Conjugacy in polycyclic groups.
\newblock {\em Algebra and Logic}, 8(6):404--411, 1969.

\bibitem{schmidt1924}
O.~Schmidt.
\newblock Uber gruppen, deren samtliche teiler spezielle gruppen sind.
\newblock {\em Matematicheskii Sbornik}, 31(3):366--372, 1924.

\bibitem{redei1956endlichen}
L.~R{\'e}dei.
\newblock Die endlichen einstufig nichtnilpotenten gruppen.
\newblock {\em Publ. Math. Debrecen}, 4(1):303--324, 1956.

\bibitem{ballester2005finite}
A.~Ballester-Bolinches, R.~Esteban-Romero, and D.~Robinson.
\newblock On finite minimal non-nilpotent groups.
\newblock {\em Proceedings of the American Mathematical Society},
  133(12):3455--3462, 2005.

\bibitem{zarrin2012}
M.~Zarrin.
\newblock A generalization of schmidt’s theorem on groups with all subgroups
  nilpotent.
\newblock {\em Archiv der Mathematik}, 99(3):201--206, 2012.

\bibitem{BRL2017}
A.~Ballester-Bolinches, R.~Esteban~Romero, and J.~Lu.
\newblock On finite groups with many supersoluble subgroups.
\newblock {\em Archiv der Mathematik}, 109(1):3--8, 2017.

\bibitem{doerk2011}
K.~Doerk and T.~O. Hawkes.
\newblock {\em Finite soluble groups}, volume~4.
\newblock Walter de Gruyter, 2011.

\bibitem{navarro2017curso}
G.~Navarro.
\newblock {\em Un curso de {\'a}lgebra}, volume~56.
\newblock Universitat de Val{\`e}ncia, 2017.

\bibitem{milne1996group}
J.~S. Milne.
\newblock Group theory.
\newblock {\em Course Notes}, 1996.

\bibitem{dummit1991abstract}
D.~S. Dummit and R.~M. Foote.
\newblock {\em Abstract algebra}, volume 1999.
\newblock Prentice Hall Englewood Cliffs, NJ, 1991.

\bibitem{mac1999algebra}
S.~Mac~Lane and G.~Birkhoff.
\newblock {\em Algebra}, volume 330.
\newblock American Mathematical Soc., 1999.

\bibitem{abmw2011}
P.~B. Acosta-Hum{\'a}nez, J.~J. Morales-Ruiz, and J.-A. Weil.
\newblock Galoisian approach to integrability of schr{\"o}dinger equation.
\newblock {\em Reports on Mathematical Physics}, 67(3):305--374, 2011.

\bibitem{athesis2009}
P.~B. Acosta-Hum{\'a}nez.
\newblock {\em Galoisian Approach to Supersymmetric Quantum Mechanics: The
  integrability analysis of the Schr{\"o}dinger equation by means of
  differential Galois theory}.
\newblock VDM Publishing, 2010.

\bibitem{manz1993}
O.~Manz and T.R. Wolf.
\newblock {\em Representations of solvable groups}.
\newblock Number 185. Cambridge University Press, 1993.

\bibitem{nicks2012}
R.~Nicks and G.~P. Parry.
\newblock On symmetries of crystals with defects related to a class of solvable
  groups (s1).
\newblock {\em Mathematics and mechanics of solids}, 17(6):631--651, 2012.

\bibitem{frattini1885intorno}
G.~Frattini.
\newblock {\em Intorno alla generazione dei gruppi di operazioni}.
\newblock Tipografia della R. Accademia dei Lincei, Propriet{\`a} del Cav. V.
  Salviucci, 1885.

\bibitem{charris2}
J.~A. Charris~Casta{\~n}eda, B.~Aldana~G{\'o}mez, and P.~B. Acosta-Hum\'anez.
\newblock {\em Algebra. Fundamentos, Grupos, Anillos, Cuerpos y Teor{\'\i}a de
  Galois}.
\newblock Bogot{\'a}: Academia Colombiana de Ciencias Exactas, F{\'\i}sicas y
  Naturales., 2013.

\bibitem{h2013I}
B.~Huppert.
\newblock {\em Endliche gruppen I}, volume 134.
\newblock Springer-verlag, 2013.

\bibitem{r2012}
D.~J. Robinson.
\newblock {\em A Course in the Theory of Groups}, volume~80.
\newblock Springer Science \& Business Media, 2012.

\bibitem{whvw1982}
M.~Weinstein, H.~G. Bray, W.~E. Deskins, D.~Johnson, J.~F. Humphreys, B.~M.
  Puttaswamaiah, P.~Venzke, and G.~L. Walls.
\newblock Between nilpotent and solvable.
\newblock {\em New Jersey}, 1982.

\bibitem{thompson1973nonsolvable}
J.~Thompson.
\newblock Nonsolvable finite groups all of whose local subgroups are solvable,
  iv.
\newblock {\em Pacific Journal of Mathematics}, 48(2):511--592, 1973.

\bibitem{dickson1959linear}
L.~E. Dickson and J.~Gillis.
\newblock Linear groups with an exposition of the galois field theory.
\newblock {\em Physics Today}, 12(10):52, 1959.

\bibitem{dickson1899report}
L.~E. Dickson.
\newblock Report on the recent progress in the theory of linear groups.
\newblock {\em Bulletin of the American Mathematical Society}, 6(1):13--27,
  1899.

\bibitem{huppert1954}
B.~Huppert.
\newblock Normalteiler und maximale untergruppen endlicher gruppen.
\newblock {\em Mathematische Zeitschrift}, 60(1):409--434, 1954.

\bibitem{lu2020finite}
J.~Lu and J.~Wang.
\newblock Finite groups with given quantitative surpersoluble subgroups.
\newblock {\em Journal of Algebra and Its Applications}, 19(04):2050077, 2020.

\bibitem{conway1985}
J.~H. Conway, R.~T. Curtis, S.~P. Norton, R.~A. Parker, and R.~A. Wilson.
\newblock {\em ATLAS of finite groups}.
\newblock Oxford University Press, Eynsham, 1985.

\bibitem{acostateoremas}
P.~B. Acosta-Hum{\'a}nez.
\newblock Teoremas de isomorf{\'\i}a en grupos diedros.
\newblock {\em Lecturas Matem\'aticas}, 24:123--136, 2003.

\bibitem{charris1}
J.~A. Charris~Casta{\~n}eda, B.~Aldana~G{\'o}mez, and P.~B. Acosta-Hum\'anez.
\newblock {\em Algebra I. Fundamentos y teor{\'\i}a de los grupos}.
\newblock Bogot{\'a}: Academia Colombiana de Ciencias Exactas, F{\'\i}sicas y
  Naturales., 2005.

\end{thebibliography}

\end{document}